\documentclass[a4paper,12pt]{article}

\usepackage[utf8]{inputenc}
\usepackage[T1]{fontenc}
\usepackage{amsmath,amssymb,amsfonts,amsthm,mathrsfs}
\usepackage{geometry}
\usepackage{graphicx}
\usepackage{float}
\usepackage[french]{babel}
\usepackage{tikz}
\usetikzlibrary{arrows}
\usetikzlibrary{decorations.markings}

\newenvironment{resume}[1]{
	\begin{list}{}{
		\setlength{\leftmargin}{1cm}
		\setlength{\rightmargin}{1cm}
	}\item[]
	{\bf #1}
	}{\end{list}}

\theoremstyle{plain}

\newtheorem{thm}{Théorème}
\newtheorem{lem}{Lemme}

\newtheorem*{thm*}{Théorème}
\newcommand{\mot}{}
\newtheorem*{thmref_interne}{\mot{}}

\newcommand\mb{\mathbb}
\newcommand\mc{\mathcal}
\newcommand\mf{\mathfrak}
\newcommand\mr{\mathrm}
\newcommand\ms{\mathscr}

\newcommand\C[1]{\mb{C}^{#1}}
\newcommand\OC[1]{\mc{O}_{(\C{#1},0)}}

\newcommand\Cg[1]{(\C{#1},0)}

\title{
Décomposition en éléments simples des formes méromorphes formelles fermées
}

\author{Olivier Thom}
\date{\today}

\begin{document}
\maketitle

\begin{resume}{Résumé :}
On prouve que toute $1$-forme méromorphe formelle fermée admet une décomposition "en éléments simples", ce qui permet en particulier de définir une notion de résidu pour les formes méromorphes formelles fermées qui étend la notion définie pour les formes usuelles.
\end{resume}

\begin{resume}{Abstract:}
We show that any closed formal meromorphic $1$-form admits a "partial fraction decomposition", which allows us in particular to define a notion of residue for closed formal meromorphic forms which extends the notion defined for usual forms.
\end{resume}

Dans toute la suite, on note $\mc{O}_n = \OC{n}$ l'anneau des germes de fonctions holomorphes à l'origine de $\mb{C}^n$ et $\Omega^1 = \Omega^1(\Cg{n})$ l'ensemble des germes de $1$-formes holomorphes sur $\Cg{n}$.
On désigne par $\widehat{\mc{O}}_n = \mb{C}[[x_1,\ldots,x_n]]$ et $\widehat{\Omega}^1 = \Omega^1\otimes \widehat{\mc{O}}_n$ leurs complétés formels.
Notons aussi $\widehat{\mc{M}}_n= \mr{Frac}(\widehat{\mc{O}}_n)$ et $\widehat{\Omega}_{\mc{M}}^1 = \Omega^1\otimes \widehat{\mc{M}}_n$.
Par analogie, on pourra appeler "fonction méromorphe formelle" un élément de $\widehat{\mc{M}}_n$ et "forme méromorphe formelle" un élément de $\widehat{\Omega}_{\mc{M}}^1$.

Soit $\omega$ un germe de $1$-forme méromorphe fermée à l'origine de $\mb{C}^n$ ; on peut écrire $\omega = \frac{\eta}{f_1^{n_1+1}\ldots f_p^{n_p+1}}$ où $\eta\in \Omega^1(\mb{C}^n,0)$, les $f_i$ sont des éléments irréductibles de $\mc{O}_n$, et les $n_i+1\in \mb{N}$ sont les multiplicités des pôles $f_i=0$.
Dans \cite{cema_asterisque}, les auteurs prouvent le résultat de décomposition "en éléments simples" suivant :
\[
\omega = \sum_i{\lambda_i \frac{df_i}{f_i}} + d \left(\frac{F}{f_1^{n_1}\ldots f_p^{n_p}} \right),
\]
où $F$ est un germe de fonction holomorphe et les $\lambda_i\in \mb{C}$ sont les résidus de $\omega$ qui se calculent en intégrant $\omega$ le long d'un petit chemin autour du pôle $f_i=0$.

L'analogue formel de cette décomposition a été conjecturé dans \cite{cgb} car il n'existe pas de référence pour ce fait.
La principale différence avec le cas convergent est qu'on ne peut pas intégrer une forme formelle le long d'un chemin pour utiliser le théorème des résidus.
On se propose dans ce papier de donner une preuve de cette décomposition dans le cas formel : plus précisément, on prouve le théorème suivant :

\begin{thm}
\label{thm1}
Soient $\eta\in \widehat{\Omega}^1$, $f_1,\ldots,f_p\in \widehat{\mc{O}}_n$ irréductibles et $n_1,\ldots,n_p$ des entiers tels que $\omega:=\frac{\eta}{f_1^{n_1+1}\ldots f_p^{n_p+1}}$ soit fermée.
Alors il existe des nombres $\lambda_1,\ldots,\lambda_p$ et une série formelle $F\in \widehat{\mc{O}}_n$ tels que
\[
\omega = \sum_i \lambda_i \frac{df_i}{f_i} + d \left(\frac{F}{f_1^{n_1}\ldots f_p^{n_p}} \right).
\]
\end{thm}

Le cas $n=1$ est bien connu puisqu'une $1$-forme méromorphe formelle est formellement conjuguée à une $1$-forme méromorphe ordinaire.
Dans \cite{cgb}, les auteurs donnent un plan de preuve en dimension $2$ basée sur la résolution des singularités du dénominateur ; on propose en partie \ref{sec_dim2} une autre preuve basée sur un argument cohomologique, toujours en dimension $2$.

Pour le cas général, on passera de la dimension $n$ à la dimension $2$ en fixant un faisceau de $2$-plans linéaires générique $\mc{H} = \{H_t\}_{t\in \mb{P}^{n-2}}$.
On pourra alors utiliser le cas $n=2$ sur chaque plan $H_t$, ce qui nous fournira en particulier des résidus $\lambda_i(t)$.
Une étape délicate à traiter est le problème que cette méthode ne donne aucune information sur la régularité des fonctions $\lambda_i(t)$.
On traitera ce point en partie \ref{sec_famille} en montrant qu'en dimension $2$, on peut en fait décomposer une famille analytique $\omega_t$ en famille, ce qui permettra de prouver que chaque $\lambda_i(t)$ est analytique en $t$.

Le but de cette étape est de pouvoir dériver les fonctions $\lambda_i(t)$, ce qui permet d'utiliser l'hypothèse que la forme $\omega$ dont on était parti est fermée pour conclure (cf. section \ref{sec_fin}).

\section{Le cas de la dimension $2$}
\label{sec_dim2}

\begin{lem}
\label{lem_coho}
Le théorème est vrai si les $f_i$ convergent.
\end{lem}

\begin{proof}
On considère $f= f_1^{n_1}\ldots f_p^{n_p}$, $f^+= f_1^{n_1+1}\ldots f_p^{n_p+1}$ et les faisceaux $\frac{1}{g}\ms{F}$ pour $g=f,f^+$ et $\ms{F} = \mc{O}_n, \widehat{\mc{O}_n}, \Omega^1, \widehat{\Omega}^1$.

On note $H^0(\frac{1}{g}\ms{F})$ les sections globales du faisceau $\frac{1}{g}\ms{F}$, $Z^1 \left(\frac{1}{f^+}\Omega^1 \right)$ (resp. $Z^1 \left(\frac{1}{f^+}\widehat{\Omega}^1 \right)$) le faisceau des $1$-formes méromorphes fermées (resp. méromorphes formelles fermées) de pôles (contenus dans) $f^+$, et $H^1 \left(\frac{1}{f^+} \Omega^1 \right) = Z^1 \left( \frac{1}{f^+}\Omega^1 \right) / dH^0 \left(\frac{1}{f}\mc{O}_n \right)$ (resp. $H^1 \left(\frac{1}{f^+} \widehat{\Omega}^1 \right)= \ldots$) le quotient des formes fermées par les formes exactes.
On a la suite exacte :
\[
0 \rightarrow H^0 \left(\frac{1}{f}\mc{O}_n \right) \rightarrow Z^1 \left(\frac{1}{f^+}\Omega^1 \right) \rightarrow H^1 \left(\frac{1}{f^+}\Omega^1 \right) \rightarrow 0,
\]
d'où, par platitude de $\widehat{\mc{O}}_n$ sur $\mc{O}_n$ (voir \cite{serre_gaga}),
\[
0 \rightarrow H^0 \left(\frac{1}{f}\widehat{\mc{O}}_n \right) \rightarrow Z^1 \left(\frac{1}{f^+}\widehat{\Omega}^1 \right) \rightarrow H^1 \left(\frac{1}{f^+}\Omega^1 \right)\otimes \widehat{\mc{O}}_n \rightarrow 0.
\]
On en déduit que
\[
H^1 \left(\frac{1}{f^+}\widehat{\Omega}^1 \right) = H^1 \left(\frac{1}{f^+}\Omega^1 \right) \otimes \widehat{\mc{O}}_n.
\]
Or, le théorème \ref{thm1} est vrai sur $\mc{O}_n$ d'après \cite{cema_asterisque} donc $H^1(\frac{1}{f^+}\Omega^1)$ est engendré par les $df_i/f_i$ ; il en résulte que $H^1(\frac{1}{f^+}\widehat{\Omega}^1)$ est aussi engendré par les $df_i/f_i$, ce qui implique le résultat.
\end{proof}

\begin{lem}
\label{lem_n2}
Le théorème est vrai pour $n=2$.
\end{lem}

\begin{proof}
En dimension $2$, la fonction $g=f_1 \ldots f_p$ est de détermination finie donc formellement difféomorphe à une fonction convergente d'après \cite{tougeron_determination_finie} (voir aussi \cite{mather2}).
On peut donc supposer que l'ensemble $\{g=0\} = \cup \{f_i=0\}$ est convergent, donc que tous les $\{f_i=0\}$ sont convergents.
Il existe alors des unités $u_i\in \widehat{\mc{O}}_n$ telles que les $u_if_i$ soient convergents.
Comme $\frac{d(u_if_i)}{u_if_i} = d\mr{log}(u_i) + \frac{df_i}{f_i}$ et que $\mr{log}(u_i)$ est bien défini puisque $u_i$ est une unité, on peut se ramener au cas où les $f_i$ convergent et conclure par le lemme \ref{lem_coho}.
%
%
\end{proof}

\section{Décomposition en famille}
\label{sec_famille}

Dans cette partie, on considère $\omega_t$ une famille de $1$-formes méromorphes formelles fermées en dimension $2$ dépendant analytiquement d'un paramètre $t\in\Cg{r}$, ie. $\omega_t = \frac{\eta_t}{f_{1,t}^{n_1+1}\ldots f_{p,t}^{n_p+1}}$ où les coefficients de $\eta_t$ et des $f_{i,t}$ dépendent analytiquement de $t$.
On cherchera à étudier la dépendance en $t$ des différents objets intervenant dans la preuve présentée en partie \ref{sec_dim2}.

Pour cela, précisons le contexte algébrique.
Notons $A=\mb{C}\{t\}[[x,y]]$ l'anneau des séries formelles $\sum_{ij}{a_{ij}(t)x^iy^i}$ où chaque $a_{ij}(t)$ converge sur un petit disque $D_{ij}$, et $A_{t_0} = \mb{C}[[x,y]]$ la restriction au plan $t=t_0$ des séries de $A$ dont les coefficients convergent en $t_0$.
Posons $g(t,x,y) = f_{1,t}(x,y)\ldots f_{p,t}(x,y) \in A$ et $g_t(x,y)=g(t,x,y)\in A_t$ ; on s'intéressera aux idéaux de $A$ donnés par $I=\langle \partial_{x}g, \partial_{y}g \rangle$ et $\mf{m} = \langle x,y \rangle$ et aux idéaux de $A_t$ donnés par $I_t = \langle \partial_{x}g_t, \partial_{y}g_t \rangle$ et $\mf{m}_t = \langle x,y\rangle$.

\begin{lem}
\label{lem_famille}
Soit $\omega_t$ une famille analytique de $1$-formes méromorphes formelles fermées en dimension $2$.
On suppose que les séries $f_{i,t}$ sont irréductibles et deux à deux premières entre elles pour tout $t$ proche de $0$.

Alors on peut décomposer les formes $\omega_t$ en famille au voisinage d'un paramètre $t_0$ générique, ie. il existe des fonctions analytiques $\lambda_i(t)$ et une famille analytique de séries formelles $F_t\in \widehat{\mc{O}}_n$ définies au voisinage de $t_0$ telles que
\[
\omega_t = \sum_i \lambda_i(t) \frac{df_{i,t}}{f_{i,t}} + d \left(\frac{F_t}{f_{1,t}^{n_1}\ldots f_{p,t}^{n_p}} \right).
\]
\end{lem}

\begin{proof}
Commençons par étudier la dépendance en $t$ du difféomorphisme formel conjugant $g(t,\cdot,\cdot)$ à un polynôme dans la preuve du lemme \ref{lem_n2}.

Considérons la décomposition de $I$ en composantes primaires $I = \cap J_j$.
Comme les $f_{i,t}$ sont irréductibles, chaque $I_t$ contient une puissance de $\mf{m}_t$ donc aucune des composantes $J_j$ n'a de support contenant strictement la droite $D$ d'équation $\{x=y=0\}$.
Ainsi, il y a un des $J_j$ (disons $J_0$) dont le support est cette droite, et tous les autres supports n'intersectent $D$ qu'en des sous-ensembles stricts.
Il s'ensuit qu'au voisinage d'un point $t_0$ générique, $I$ est égal à $J_0$ donc est $\mf{m}$-primaire.
Or, cela implique qu'il contient une puissance de $\mf{m}$ : il existe $k$ tel que $\mf{m}^k\subset I$ au voisinage d'un point $t_0$ générique.

Pour simplifier dans la suite, on considère que $t_0=0$.
Écrivons $g = g_0 + r$ où $g_0$ est polynomial dans les variables $x,y$, de degré total majoré par $k$, et $r\in \mf{m}^{k+1}$.
Posons aussi $g_s = g_0 + sr$ pour $s$ dans $\mb{C}$.

On doit maintenant incorporer le paramètre $s$ comme variable et réitérer le raisonnement précédent.
Posons $t'=(t,s)$ et $A',I',\mf{m}', g'$ les objets obtenus en remplaçant le paramètre $t$ par $t'$ dans les définitions précédentes.
Notons aussi $I_s$ l'idéal $\langle \partial_{x}g_s, \partial_{y}g_s\rangle\subset A$.
Écrivons la décomposition de $I'$ en composantes primaires $I'=\cap J'_j$.
On remarque que $\partial_{x}r\in I$ et $\partial_{y}r\in I$ de sorte que $I_s = I$ pour tout $s$.
Si l'un des $J'_j$ n'était pas invariant le long de la fibration donnée par $s$ au voisinage de $t=0$, on aurait $I_{s_0}\subsetneq I_s$ au voisinage du paramètre $s_0$ où le support de $J'_j$ intersecte la droite $\{x=y=t=0\}$.
Ainsi $J'_j = J_j\otimes \mb{C}\{s\}$ et $I'$ est en fait $\mf{m}'$-primaire au voisinage de la droite $\{x=y=t=0\}$.

Tout ceci nous assure qu'il existe des éléments $a_1,a_2\in A'=\mb{C}\{t,s\}[[x,y]]$ qui s'annulent en $x=y=0$ tels que
\[
r = a_1 \partial_{x}g_s + a_2 \partial_{y}g_s
\]
dans l'anneau $A'$.
Introduisons enfin le champ de vecteurs $X = -a_1 \partial_{x} - a_2 \partial_{y}+\partial_{s}$, vu comme champ de vecteurs dans l'espace des variables $s,x,y$, dépendant analytiquement du paramètre $t$ ; son flot au temps $s_0$, noté $\varphi^{s_0}$, est un difféomorphisme formel entre les plans $s=0$ et $s=s_0$, dépendant analytiquement du paramètre $t$.
Alors la dérivée dans la direction $X$ de $g'$ vaut $\partial_{X}g'=0$ par construction.
Donc $g'$ est invariant sous le flot $\varphi^s$ et donc $g_0\circ \varphi^1 = g$.

On voit donc dans cette construction que les polynômes $g_0$ et les difféomorphismes $\varphi^1$ dépendent analytiquement du paramètre $t$.

Reprenons l'analyse de la dépendance en $t$ de la preuve du lemme \ref{lem_n2}.
Quitte à composer pour tout $t$ par le difféomorphisme trouvé, on peut supposer que $g_t$ est polynomial pour tout $t$.
Alors tous les $\{f_{i,t}=0\}$ sont convergents, et comme $f_{i,t}$ dépend analytiquement de $t$, les $\{f_{i,t}=0\}$ dépendent aussi analytiquement de $t$.
Il existe donc des équations convergentes $u_{i,t}f_{i,t}$ dépendant analytiquement de $t$.
C'est ainsi que l'on se ramène au cas où les $f_{i,t}$ convergent.

Maintenant, il faut étudier la dépendance en $t$ de la preuve du lemme \ref{lem_coho}.
On peut en fait réutiliser exactement les mêmes idées en remplaçant systématiquement chaque faisceau $\ms{F}$ par le faisceau $\ms{F}'$ des familles d'éléments de $\ms{F}$ dont les coefficients dépendent analytiquement de paramètres $t$.
Pour faire fonctionner la démonstration il faut alors considérer la dérivation $d$ uniquement dans les variables spaciales, et utiliser une notion de "platitude en famille" ; plus précisément, on a besoin du fait que la complétion $\widehat{\mc{O}}'_n$ de $\mc{O}'_n$ par rapport à l'idéal $\mf{m}=\langle x,y\rangle$ soit plate.
Ce fait peut être trouvé dans \cite{eisenbud}.
On est ainsi ramené à prouver le lemme dans le cas convergent. Dans ce cas, les résidus sont définis par des intégrales donc dépendent analytiquement des paramètres, ce qui termine la preuve.
\end{proof}

\section{La preuve en dimension $n$}
\label{sec_fin}

Commençons par prouver quelques lemmes dont nous aurons besoin.

\begin{lem}
\label{lem_dlog}
Si $f_i$ sont des séries formelles irréductibles deux à deux premières entre elles et $\lambda_i\in \mb{C}$, alors $\sum \lambda_i \frac{df_i}{f_i}$ admet une primitive méromorphe formelle si et seulement si $\lambda_i=0$ pour tout $i$.
\end{lem}

\begin{proof}
Supposons que $\sum \lambda_i \frac{df_i}{f_i} = dF$ pour une fonction méromorphe formelle $F$.
Alors les pôles de $dF$ sont d'ordre au moins deux donc $\lambda_i=0$ pour tout $i$.
\end{proof}

Le prochain lemme examine comment se comportent les résidus d'une forme globale restreinte aux fibres d'une fibration régulière.
On va avoir besoin de distinguer les opérateurs de différentiations en considérant $t$ comme un paramètre ou comme une variable, donc pour le prochain lemme, on introduit momentanément les notations suivantes.

Dans un système de coordonnées $(t,x)=(t_1,\ldots,t_{n-2},x_1,x_2)$, on notera $d$ l'opérateur de différentiation pour les variables $x$ et $D$ l'opérateur de différentiation totale (pour les variables $t$ et $x$).
Ces opérateurs pourront agir sur les séries ou les $1$-formes, par exemple $D(f) = \sum \partial_{t_j}f dt_j + \sum \partial_{x_j}f dx_j$.

\begin{lem}
\label{lem_semilocal}
Soit $\omega$ une $1$-forme méromorphe formelle fermée sur un voisinage formel $V=U\times \Cg{2}$ d'un ouvert $U\subset \mb{C}^{n-2}$.
Soit $\mc{H} = \{H_t\}$ une fibration régulière de dimension $2$ sur $V$ transverse à $U$.
Notons $i_t : H_t \hookrightarrow V$ et $\omega_t = i_t^*\omega$.

Supposons qu'il existe des fonctions méromorphes formelles $s_t$ et des résidus $\lambda_{i,t}$ tels que
\[
\omega_t = \sum_{i=1}^k\lambda_{i,t} \frac{df_{i,t}}{f_{i,t}} + d(s_t)
\]
pour tout $t$ et que les $f_{i,t}$ sont irréductibles et deux à deux premiers entre eux.
Alors il existe un ouvert $U'\subset U$ dense sur lequel $\lambda_{i,t}$ est constante en $t$ et si on pose $F_i(t,\cdot) = f_{i,t}(\cdot)$ et $S(t,\cdot) = s_t(\cdot)$, il existe des fonctions méromorphes formelles $g_j(t)$ telles que
\[
\omega = \sum_{i=1}^k\lambda_i \frac{DF_i}{F_i} + DS + \sum_{j=1}^{n-2} g_j(t)dt_j.
\]
\end{lem}

\begin{proof}
%
Tout d'abord, notons que le lemme \ref{lem_famille} s'applique : les résidus $\lambda_{i,t}$ dépendent analytiquement de $t$ sur un ouvert dense $U'\subset U$ (et donc la série $s_t$ aussi si on choisit $s_t(0)=0$).
Notons $\tilde{\omega}$ la section de $T^*\mc{H}$ telle que $i_t^*\tilde{\omega}=\omega_t$ pour tout $t$ et $\mu = \sum\lambda_{i,t} \frac{DF_i}{F_i} + DS$ de sorte que $i_t^*\mu = \omega_t$.
On obtient donc $\omega = \tilde{\omega} + \sum p_j dt_j$ pour des fonctions méromorphes formelles $p_j$ et $\mu = \tilde{\omega} + \sum q_j dt_j$ pour 
\[
q_j = \sum_i \lambda_{i,t} \frac{\partial_{t_j}F_i}{F_i} + \partial_{t_j}S.
\]

D'un autre côté, $\omega$ est fermée donc $D \omega=0$, or
\begin{align*}
D \omega &= d \omega_t + \sum_j dt_j \wedge \partial_{t_j}\tilde{\omega} + \sum_j Dp_j \wedge dt_j\\
&= \sum_j dt_j \wedge \left( \sum_i \lambda_{i,t} \partial_{t_j}d \mr{log}(F_i) + \sum_i (\partial_{t_j} \lambda_{i,t}) d\mr{log}(F_i) + \partial_{t_j}dS \right) + \sum_j Dp_j \wedge dt_j.
\end{align*}
Ainsi, les coefficients de $dt_j\wedge dx_1$ et de $dt_j\wedge dx_2$ donnent $dp_j = d (\sum_i \lambda_{i,t} \partial_{t_j}\mr{log}(F_i) + \sum(\partial_{t_j} \lambda_{i,t})\mr{log}(F_i) + \partial_{t_j}S)$ et donc il existe une fonction $g_j(t)$ telle que 
\[
p_j = \sum_i \lambda_{i,t} \partial_{t_j}\mr{log}(F_i) + \sum_i(\partial_{t_j} \lambda_{i,t}) \mr{log}(F_i) + \partial_{t_j}S + g_j(t).
\]
Or pour chaque $t_0$ fixé, les fonctions $F_i(t_0,\cdot)$ s'annulent en $0$, sont irréductibles et d'après le calcul précédent, $\sum(\partial_{t_j} \lambda_{i,t})\vert_{t=t_0}\mr{log}(F_i(t_0,\cdot))$ est méromorphe.
On en déduit d'après le lemme \ref{lem_dlog} que $(\partial_{t_j}\lambda_{i,t})\vert_{t=t_0}=0$ pour tous $i,j,t_0$ et que
\[
\omega = \mu + \sum_j g_j(t)dt_j.
\]
\end{proof}

\begin{proof}[Preuve du théorème \ref{thm1}]
Pour se ramener à la dimension $2$, fixons une droite générique $A$ et considérons le faisceau de $2$-plans $\mc{H}=\{H_t\}_{t\in \mb{P}^{n-2}}$ d'axe $A$ : si $A=\mr{Vect}(a)$, $H_t = \mr{Vect}(a,b_t)$ où $b_t$ décrit un supplémentaire de $A$.
Notons $i_t : H_t\hookrightarrow \Cg{n}$ et $\omega_t = i_t^*\omega$ ; la généricité du faisceau $\mc{H}$ garantit que $\omega_t\neq 0$ pour tout $t$.
On peut donc appliquer le théorème sur $H_t$ ; cependant, les $f_i\vert_{H_t}$ ne sont peut-être plus irréductibles.
Ainsi, si $f_i\vert_{H_t} = f_{i,1}(t)\ldots f_{i,k}(t)$, le théorème nous livre $k$ résidus $\lambda_{i,1}(t),\ldots,\lambda_{i,k}(t)$ ; montrons que ces résidus ne dépendent pas de $t$ et que $\lambda_{i,j}=\lambda_{i,1}$ pour tout $j$.

On fait un éclatement le long de $A$ : notons $X$ la variété éclatée et $b$ l'éclatement.
Remarquons que $X$ est isomorphe à $(\mb{C},0)\times Y$ si $Y$ est l'éclaté de $(\mb{C}^{n-1},0)$ en $0$.
Soit $E=b^{-1}(\{0\})$ et notons $Z$ l'ensemble des valeurs de $t\in E$ non génériques relativement à $\omega$ et à l'éclatement $b$. 
Plus précisément, un point $t\in E$ appartient à $Z$ s'il satisfait une des conditions suivantes :
\begin{enumerate}
\item deux $f_{i,j}(t)$ coïncident ;
\item un des $f_{i,j_0}(t_0)$ se décompose en plusieurs $f_{i,j}(t)$ pour $t$ proche de $t_0$ ;
\item $t$ est un paramètre non générique au sens des lemmes précédents.
\end{enumerate}
Introduisons également l'ouvert de Zariski $U' = E\setminus Z$.

Alors les séries $f_{i,j}$ et les résidus $\lambda_{i,j}$ sont bien définis localement sur $U'$.
Par le lemme \ref{lem_semilocal}, on peut écrire
\[
b^*\omega = \sum_{i,j} \lambda_{i,j} \frac{d(b^*f_{i,j})}{b^*f_{ij}} + dS + \sum g_j(t)dt_j
\]
au voisinage de tout point de $U'$.
En particulier, le résidu $\lambda_{i,j}$ est constant le long de la branche $b^*f_{i,j}$ en dehors de $Z$.
Les pôles de $b^* \omega$ viennent tous de pôles de $\omega$ donc ils ne sont pas tangents aux fibres de l'éclatement générique $b$.
Ainsi $\sum g_j(t)dt_j$ n'a pas de résidus et admet donc une intégrale première (on utilise ici le cas convergent : les $t_j$ étant les variables de l'éclatement, on a affaire à une forme convergente).
Quitte à changer $S$, on peut ainsi supposer que $g_j=0$ pour tout $j$.

En dehors de $Z$, on peut suivre une branche $b^*f_{i,j}$ quand $t$ varie et on obtient donc une application de monodromie $\rho : \pi_1(U') \rightarrow \mf{S}([[1,k]])$ qui à un lacet $\gamma$ basé en $t_0$ associe l'indice $j$ tel que $b^*f_{i,j}(t_0) = \gamma^*(b^*f_{i,1}(t_0))$ à constante multiplicative près.
Supposons que $\rho$ ne soit pas transitive : il existe $J\subset [[1,k]]$ stable.
Alors $f_{i,J}:=\prod_{j\in J} b^*f_{i,j}$ donne une hypersurface $Y_J$ invariante par $\rho$, donc qui provient d'une surface $b(Y_J)$ ; comme $b(Y_J)\subset \{f_i=0\}$, l'équation $g_J$ de $b(Y_j)$ divise $f_i$, mais n'est ni inversible ni égale à $f_i$ à constante près.

Ceci contredit l'irréductibilité de $f_i$ donc $\rho$ est transitive et $\lambda_{i,j} = \lambda_{i,1}$ pour tout $j$, notons-le $\lambda_i$.
On obtient ainsi
\[
b^*\omega = \sum_i \lambda_i \frac{d(b^*f_i)}{b^*f_i} + dS
\]
au voisinage de $U'$.
On remarque alors que $dS = b^*\omega-\sum_i \lambda_i \frac{d(b^*f_i)}{b^*f_i}$ a des pôles bien précis, donc que pour des entiers $k_i$ assez grands, $d(b^*(f_1^{k_1}\ldots f_p^{k_p})S)$ est localement bornée (ou, plus précisément : les dérivées des coefficients de la série formelle $b^*(f_1^{k_1}\ldots f_p^{k_p})S$ sont localement bornés) sur $E$.
Il s'ensuit que $S$ se prolonge en une fonction méromorphe formelle définie sur $X$ tout entier.

Cette formule passe ainsi à l'éclatement pour donner
\[
\omega = \sum_i \lambda_i \frac{df_i}{f_i} + dS.
\]
\end{proof}

\bibliography{mybib}{}

\begin{thebibliography}{1}

\bibitem{cgb}
{\sc Cerveau, D., and Garba\:Belko, D.}
\newblock Théorèmes de {B}orel avec contraintes.
\newblock {\em J. Singul. 17\/} (2018), 245--266.

\bibitem{cema_asterisque}
{\sc Cerveau, D., and Mattei, J.-F.}
\newblock {\em Formes intégrables holomorphes singulières}.
\newblock Astérisque 97. Soc. Math. France, 1982.

\bibitem{eisenbud}
{\sc Eisenbud, D.}
\newblock {\em Commutative Algebra with a View Toward Algebraic Geometry}.
\newblock Graduate Texts in Mathematics. Springer, 1995.

\bibitem{mather2}
{\sc Mather, J.~N.}
\newblock Stability of ${C}^\infty$ mappings. {II}.
\newblock {\em Ann. of Math. 89\/} (1969), 254--291.

\bibitem{serre_gaga}
{\sc Serre, J.~P.}
\newblock Géométrie algébrique et géométrie analytique.
\newblock {\em Ann. Inst. Fourier 6\/} (1955-56), 1--42.

\bibitem{tougeron_determination_finie}
{\sc Tougeron, J.-C.}
\newblock Équivalence des idéaux de fonctions différentiables.
\newblock {\em C. R. Acad. Sci. Paris Sér. A-B 262\/} (1966), 563--565.

\end{thebibliography}
\bibliographystyle{acm}

\textsc{IMPA, Estrada Dona Castorina, 110, Horto, Rio de Janeiro, Brasil}

\textit{Email :} olivier.thom@impa.br

\end{document}